\documentclass[reqno]{amsart}

\usepackage{amssymb}
\usepackage[british]{babel}
\usepackage{tikz-cd}
\usepackage{dsfont}

\newcommand{\ppmod}[1]{\!\pmod{#1}} 
\makeatletter
\newcommand{\dpmod}[1]{{\@displayfalse\ppmod{#1}}} 
\makeatother

\theoremstyle{plain}
\newtheorem{theorem}{Theorem}[section]
\newtheorem{proposition}[theorem]{Proposition}
\newtheorem{lemma}[theorem]{Lemma}
\newtheorem{corollary}[theorem]{Corollary}

\theoremstyle{remark}
\newtheorem*{remark}{Remark}

\theoremstyle{definition}
\newtheorem*{definition}{Definition}
\newtheorem{example}[theorem]{Example}

\DeclareMathOperator{\End}{End}
\DeclareMathOperator{\Hom}{Hom}

\DeclareMathOperator{\Imag}{im}
\DeclareMathOperator{\Id}{Id}
\DeclareMathOperator{\ord}{ord}
\DeclareMathOperator{\Aut}{Aut}
\DeclareMathOperator{\lann}{l-ann}
\DeclareMathOperator{\rann}{r-ann}

\begin{document}

\title{Algebraic aspects of general free skew extensions of rings}

\author{Vitor O. Ferreira}
\thanks{The first and third authors were partially supported by grant 2020/16594-0, 
S\~ao Paulo Research Foundation (FAPESP), Brazil}
\address{Department of Mathematics, University of S\~ao Paulo, S\~ao Paulo, SP, 05508-090, Brazil}
\email{vofer@ime.usp.br}
  
\author{\'Erica Z. Fornaroli}
\address{Department of Mathematics, State University of Maring\'a, Maring\'a, PR, 
87020-900, Brazil}
	\email{ezancanella@uem.br}

\author{Javier S\'anchez}
\address{Department of Mathematics, University of S\~ao Paulo, S\~ao Paulo, SP, 05508-090, Brazil}
\email{jsanchez@ime.usp.br}


\subjclass[2020]{16S10, 16S36, 16W60}

\keywords{Skew free extensions, multivariate skew polynomial rings, Ore extensions}

\begin{abstract}
We consider skew free extensions of rings, also known as free multivariate skew polynomial rings,
and explore some of the algebraic aspects of this construction. We give different characterizations of such rings and present conditions for such a ring to be a domain,
to be embeddable in a series ring and to be prime.
\end{abstract}

\maketitle

\section*{Introduction}

In this article, we investigate purely algebraic aspects of certain families of freely
generated ring extensions, hereby called free skew extensions. The systematic study of
such ring extensions started with \cite{MK2019, uM2020, uM2022}, under the name
``multivariate skew polynomials'', where the case of a division ring of coefficients is
considered and the main focus was in evaluations of such polynomials.
Other occurrences of special cases of this construction
are \cite{BG2020, pG2020, BT2024}. In \cite{LM2024}, the authors consider more general
skew free extensions, with an arbitrary ring of coefficients, and the center of
such rings is described. This last article mentions that the
first appearance of free skew extensions is \cite{eK1966}, but no further developments
have been registered prior to \cite{MK2019}.

The construction of free skew extensions generalizes that of Ore extensions, or skew
polynomial rings. We recall
that given a ring \(R\), a ring homomorphism \(\sigma\colon R\to R\) and a right
\(\sigma\)-derivation \(\delta\colon R\to R\), that is, an additive map satisfying
\(\delta(ab)=\delta(a)\sigma(b) + a\delta(b)\), for all \(a,b\in R\), the Ore
extension \(R[x;\delta,\sigma]\) of \(R\) determined by \(\sigma\) and \(\delta\)
is the ring formed by polynomials
\[
\sum_{i=0}^n x^ia_i,
\]
in the indeterminate \(x\) with coefficients \(a_i\in R\), whose multiplication
satisfies the rule
\[
ax = x\sigma(a) + \delta(a),
\]
for all \(a\in R\). Skew free extensions are similar constructions involving
more than one indeterminate. A precise definition is given
in Section~\ref{sec:def}, but informally, a skew free extension \(S\) of a ring \(R\)
is a ring whose elements are polynomials of the form
\[
\sum_{w}wa_w,
\]
where \(w\) are monomials in the noncommuting indeterminates \(x_1,\dots,x_n\)
with coefficients \(a_w\) from \(R\) satisfying relations of the form
\[
ax_j = \sum_{i=1}^nx_i\sigma_{ij}(a) + \delta_{j}(a),
\]
for all \(a\in R\) and \(j=1,\dots, n\). The precise conditions on the
maps \(\sigma_{ij},\delta_j\colon R\to R\) that impose (and are consequences
of) the associativity of the multiplication of \(S\) are detailed in
Section~\ref{sec:def}. We shall see that, unlike iterated Ore extensions, skew free 
extensions behave more like free algebras than polynomial algebras. 

In Section~\ref{sec:def}, after defining general skew free extensions, we
show that not only these rings exist (using a different argument from the
one in \cite{MK2019}) but also that these rings are unique. Section~\ref{sec:tr}
gives an alternative construction using tensor rings, showing the narrow
relation that skew free extensions have with free algebras. The consideration of the problem of
when a skew free extension is a domain occupies Section~\ref{sec:domain},
where the concept of megainjectivity of a ring homomorphism is introduced and
shown to be the precise condition for the absence of zero divisors in a free skew extension.

The construction of a series ring containing a skew free extension is the
subject of Section~\ref{sec:series}. We prove that the expected condition of
local nilpotence on the derivation is sufficient to guarantee that the multiplication defined
in a skew free extension can be extended to handle infinite sums. This had already been proved
in \cite{GSZ2019} for the case of ordinary one-variable Ore extensions. Finally, in
Section~\ref{sec:prime}, we tackle the problem of characterizing prime free
skew extensions, but we must restrict to the case of triangular homomorphism.
It is shown that a known condition for Ore extensions, previously obtained in \cite{LLM1997},
has an analogy in the context of
various indeterminates.

We deviate from the convention adopted in the above cited papers of putting coefficients
on the left and, instead, place them on the right. This was done in order to present
our results in a more direct way.

\medskip

Throughout the paper, rings are associative and contain an identity element, which is preserved
by homomorphisms and acts trivially on modules.

\section{Definition, existence and uniqueness}\label{sec:def}

This section is devoted to the proof of existence and uniqueness of skew free
extensions. We also show that these rings satisfy a universal mapping property.

\medskip

Let \(R\) be a ring. An \emph{\(R\)-ring} is an ordered pair \((A,\lambda_A)\), where
\(A\) is a ring and \(\lambda_A\colon R\to A\) is a ring homomorphism. Given \(R\)-rings
\((A,\lambda_A)\) and \((B,\lambda_B)\) an \emph{\(R\)-ring homomorphism}  from \((A,\lambda_A)\) into \((B,\lambda_B)\) is a ring
homomorphism \(\varphi\colon A\to B\) such that \(\varphi\lambda_A=\lambda_B\). Clearly, if
\((A,\lambda_A)\) is an \(R\)-ring and \(I\) is an ideal of \(A\), then the quotient ring
\(A/I\) has a natural \(R\)-ring structure which makes the canonical surjective
ring homomorphism \(A\to A/I\) an \(R\)-ring homomorphism.

Recall that an \emph{\(R\)-bimodule} is an additive abelian group \(M\) which is a left
\(R\)-module, a right \(R\)-module and such that
\[
(rx) s = r (x s)
\]
for all \(r,s\in R\) and \(x\in M\). So, if \((A,\lambda_A)\)
is an \(R\)-ring, then \(A\) has a natural \(R\)-bimodule structure given by
\[
r a = \lambda_A(r)a
\quad\text{and}\quad
a  r = a\lambda_A(r),
\]
for all \(r\in R\) and \(a\in A\). This \(R\)-bimodule structure on \(A\) is such that
\begin{equation}\label{eq:rring}
r (ab) = (r a)b, \quad (ar)b=a(rb), \quad (ab) r = a(b r),
\end{equation}
for all \(r\in R\) and \(a,b\in A\). Conversely, given an \(R\)-bimodule structure on a
ring \(A\) satisfying \eqref{eq:rring}, then \(A\) becomes an \(R\)-ring via 
\[
\begin{array}{rcl}
\lambda_A\colon R & \longrightarrow & A\\ r&\longmapsto &r 1_A
\end{array}
\]
We shall switch from these two viewpoints according to convenience.

If \(A\) and \(B\) are \(R\)-rings and \(\varphi\colon A\to B\)
is a ring homomorphism that is also a left (or right) \(R\)-module homomorphism,
then \(\varphi\) is an \(R\)-ring homomorphism. 

\medskip

Let \(R\) be a ring, let \(n\) be a positive integer
and let \(\sigma\colon R\to M_n(R)\) be a ring homomorphism. 
Let \(R^n\) stand for the free left \(R\)-module of rank \(n\), hereby identified with
the set of row vectors with \(n\) entries from \(R\). An additive map \(\delta\colon R\to R^n\)
satisfying
\begin{equation}\label{eq:lsd}
\delta(rs) = \delta(r)\sigma(s)+r\delta(s),
\end{equation}
for all \(r,s\in R\), will be called a \emph{right \(\sigma\)-derivation}. The concatenation
\(\delta(r)\sigma(s)\) on the right-hand side of \eqref{eq:lsd} stands for the matrix product
of the \(1\times n\) matrix \(\delta(r)\) by the \(n\times n\) matrix \(\sigma(s)\).

Writing \(\sigma_{ij}(r)\) for the \((i,j)\)-entry of the matrix \(\sigma(r)\) and
\(\delta_j(r)\) for the \(j\)-component of the row vector \(\delta(r)\), it follows that
\(\sigma_{ij}\) and \(\delta_j\) are additive endomorphisms of \(R\) satisfying
\[
\sigma_{ij}(rs)=\sum_{k=1}^n\sigma_{ik}(r)\sigma_{kj}(s)
\]
\and
\[
\delta_j(rs) = \sum_{i=1}^n\delta_i(r)\sigma_{ij}(s)+r\delta_j(s),
\]
for all \(r,s\in R\). Moreover, 
\[
\sigma_{ij}(1)=\begin{cases} 1, &\text{if \(i=j\)}\\ 0, &\text{otherwise}\end{cases}
\quad\text{and}\quad\delta_j(1)=0.
\]

Given a ring \(S\), we shall write \(S^{\times}\) for the monoid structure on
\(S\) with operation given by the multiplication.

\begin{definition}
Let \(R\) be a ring, let \(n\) be a positive integer,
let \(\sigma\colon R\to M_n(R)\) be a ring homomorphism and let
\(\delta \colon R\to R^n\) be a right \(\sigma\)-derivation. An \(R\)-ring
\(S\) containing an \(n\)-element subset \(X=\{x_1,\dots,x_n\}\) such that 
\renewcommand{\labelenumi}{(\roman{enumi})} 
\begin{enumerate}
  \item the submonoid \(\langle X\rangle\) of \(S^{\times}\) generated by \(X\) is free on \(X\),
	\item \(S\) is a free right \(R\)-module with basis \(\langle X\rangle\), and
	\item \(rx_j=\sum_{i=1}^nx_i\sigma_{ij}(r) + \delta_j(r)\), for all \(j=1,\dots, n\) and
	\(r\in R\),
\end{enumerate} 
is called a \emph{\((\sigma,\delta)\)-skew free extension}
of \(R\) generated by \(X\).
\end{definition}

We note that since \(1_S\in\langle X\rangle\) the ring homomorphism \(R\to S\) affording
the \(R\)-ring structure on a skew free extension \(S\) of \(R\)
is injective and, therefore, \(R\) can be regarded
as a subring of \(S\).

Our first result, following the ideas in \cite[Prop.~2.3]{GW2004} and \cite[1.2.3]{MR2001}, shows that skew free extensions always exist.

\begin{theorem}\label{th:existsfe}
Let \(R\) be a ring, let \(n\) be a positive integer,
let \(\sigma\colon R\to M_n(R)\) be a ring homomorphism and let
\(\delta \colon R\to R^n\) be a right \(\sigma\)-derivation. Then there exists
a \((\sigma,\delta)\)-free skew extension of \(R\) generated by \(n\) elements.
\end{theorem}

\begin{proof}
  Let \(\mathcal{F}\) be the free monoid on the \(n\)-element set \(\{z_1,\dots,z_n\}\) and
	let \(R^{\mathcal{F}}\) denote the set of functions \(\mathcal{F}\to R\). Then 
	\(R^{\mathcal{F}}\)
	has a natural abelian group structure such that \((f+g)(m)=f(m)+g(m)\), for all
	\(f,g\in R^{\mathcal{F}}\) and \(m\in\mathcal{F}\). Let \(E\) be the ring of group
	endomorphisms of \(R^{\mathcal{F}}\), acting on the elements of \(R^{\mathcal{F}}\)
	on the right. Each \(r\in R\) defines \(\lambda(r)\in E\) by right multiplication:
	\(\bigl(f\lambda(r)\bigr)(m)=f(m)r\), for all \(f\in R^{\mathcal{F}}\) and \(m\in\mathcal{F}\). The
	map \(r\mapsto\lambda(r)\) is easily seen to be an injective ring homomorphism \(R\to E\),
	defining in \(E\) an \(R\)-ring structure. Moreover, for each \(j=1,\dots, n\), one has
	\(x_j\in E\) defined by
	\[
	(fx_j)(mz_i)=	\sigma_{ij}\bigl(f(m)\bigr)+\delta_j\bigl(f(mz_i)\bigr), \quad\text{for \(i=1,\dots,n\) and \(m\in\mathcal{F}\)}
	\]
	and
	\[
	(fx_j)(\mathbf{1})= \delta_j\bigl(f(\mathbf{1})\bigr),
	\]
	where \(\mathbf{1}\) denotes the unity element of \(\mathcal{F}\).
	
	For each \(m\in\mathcal{F}\), let \(\overline{m}\in R^{\mathcal{F}}\) be defined by
	\[
	\overline{m}(l)=\begin{cases} 1, & \text{if \(l=m\)}\\ 0, &\text{otherwise.}\end{cases}
	\]
	We note that the function \(m\mapsto\overline{m}\), from \(\mathcal{F}\) into \(R^{\mathcal{F}}\),
	is injective.
	By the definition of \(x_j\), on has
	\[
	\overline{m}x_j=\overline{mz_j},
	\]
	for all \(m\in\mathcal{F}\), and, so it follows by induction on \(k\) that
	\[
	\overline{\mathbf{1}}(x_{j_1}\dotsm x_{j_k})=\overline{z_{j_1}\dotsm z_{j_k}}.
	\]
  Hence, the submonoid \(\langle X\rangle\) of \(E^{\times}\) generated by \(X=\{x_1,\dots,x_n\}\) 
	is free on
	\(X\). In particular, there exists a monoid isomorphism \(\phi\colon \langle X\rangle \to
	\mathcal{F}\) such that \(\phi(x_j)=z_j\). Given a right linear combination
	\[
	\sum_{m\in\mathcal{F}}mr_m,
	\]
	with \(r_m\in R\), one has
	\[
  \overline{\mathbf{1}}\left(\sum_{m\in\mathcal{F}}mr_m\right) = 
	\sum_{m\in\mathcal{F}}\overline{\phi(m)}r_m,
	\]
	and this endomorphism evaluated at \(\phi(m)\) gives \(r_{\phi(m)}\). Therefore,
	the set \(\langle X\rangle\) is right linearly
	independent over \(R\) in \(E\).
	
	Let \(r\in R\). We show that, in \(E\), for each \(j=1,\dots,n\), one  has
	\(rx_j=\sum_{k=1}^nx_k\sigma_{kj}(r) + \delta_j(r)\). Indeed, given
	\(f\in R^{\mathcal{F}}\) and \(m\in\mathcal{F}\), 
	\begin{equation*}\begin{split}
	\bigl(f(rx_j)\bigr)(mz_i) & = \bigl((fr)x_j\bigr)(mz_i)\\
	& = \sigma_{ij}\bigl((fr)(m)\bigr) + \delta_j\bigl((fr)(mz_i)\bigr)\\
	& = \sigma_{ij}\bigl(f(m)r\bigr)+\delta_j\bigl(f(mz_i)r\bigr)\\
	& =\sum_{k=1}^n\sigma_{ik}\bigl(f(m)\bigr)\sigma_{kj}(r) + 
	  \sum_{k=1}^n\delta_k\bigl(f(mz_i)\bigr)\sigma_{kj}(r) + f(mz_i)\delta_j(r)\\
	& = \sum_{k=1}^n\Bigl(\sigma_{ik}\bigl(f(m)\bigr)+\delta_k\bigl(f(mz_i)\bigr)\Bigr)\sigma_{kj}(r) + f(mz_i)\delta_j(r)\\
	& = \sum_{k=1}^n(fx_k)(mz_i)\sigma_{kj}(r)+ f(mz_i)\delta_j(r)\\
	& = \sum_{k=1}^n\bigl((fx_k)\sigma_{kj}(r)\bigr)(mz_i) + \bigl(f\delta_j(r)\bigr)(mz_i)\\
	& = \sum_{k=1}^n \bigl(f(x_k\sigma_{kj})(r)+f\delta(r)\bigr)(mz_i)\\
	& = \biggl(f\Bigl(\sum_{k=1}^nx_k\sigma_{kj}(r)+\delta_j(r)\Bigr)\biggr)(mz_i).
	\end{split}\end{equation*}
	Also,
	\begin{equation*}\begin{split}
	\bigl(f(rx_j)\bigr)(\mathbf{1}) & = \bigl((fr)x_j\bigr)(\mathbf{1}) = 
	\delta_j\bigl((fr)(\mathbf{1})\bigr)=
	   \delta_j\bigl(f(\mathbf{1})r\bigr)\\
	& = \sum_{k=1}^n \delta_j\bigl(f(\mathbf{1})\bigr)\sigma_{kj}(r) + f(\mathbf{1})\delta_j(r)\\
	& = \sum_{k=1}^n (fx_k)(\mathbf{1})\sigma_{kj}(r) + f(\mathbf{1})\delta_j(r)\\
	& = \sum_{k=1}^n\bigl((fx_k)\sigma_{kj}(r)\bigr)(\mathbf{1})+ f(\mathbf{1})\delta_j(r)\\
	& = \biggl(f\Bigr(\sum_{k=1}^n x_k\sigma_{kj}(r)+\delta_j(r)\Bigl)\biggr)(\mathbf{1}).
	\end{split}\end{equation*}
	Hence, the operators \(rx_j\) and \(\sum_{k=1}^nx_k\sigma_{kj}(r) + \delta_j(r)\)
	do coincide. Besides, it follows that, in \(E\),
	\begin{equation}\label{eq:multS-pre}
	Rx_1+\dots+Rx_n\subseteq x_1R+\dots+x_nR+R.
	\end{equation}
	If \(w\in\langle X\rangle\) is such that \(w=x_{i_1}x_{i_2}\dotsm x_{i_r}\), with
\(i_k\in\{1,\dots,n\}\) for all \(k=1,\dots, r\), we say that \(w\) has
\emph{length} \(r\) and write \(\lvert w\rvert = r\). We also set \(\lvert 1\rvert = 0\).
As a consequence of \eqref{eq:multS-pre}, an inductive argument on the length \(\lvert w\rvert\) of 
	\(w\in\langle X\rangle\), gives
	\begin{equation}\label{eq:multS}
	Rw\subseteq \sum_{\lvert v \rvert \leq \lvert w \rvert} vR.
	\end{equation}
	Let \(S\) be the (free) right \(R\)-submodule of \(E\) generated by \(\langle X\rangle\).
	Given \(w_1,w_2\in\langle X \rangle\), it follows from \eqref{eq:multS} that
	\(w_1Rw_2R\subseteq S\); this proves that \(S\) is a sub-\(R\)-ring of \(E\). By what
	has been shown above, \(S\) is a \((\sigma,\delta)\)-free skew extension of \(R\) generated
	by \(x_1,\dots,x_n\). 
\end{proof}

Having dealt with the existence of free skew extensions, we now turn to uniqueness,
which will follow from the fact that these rings satisfy a universal property
(stated, in a special case, without proof in \cite[Lemma~2]{uM2022}) as the
next result shows.

\begin{proposition}\label{prop:univsfe}
Let \(R\) be a ring, let \(n\) be a positive integer,
let \(\sigma\colon R\to M_n(R)\) be a ring homomorphism, let
\(\delta \colon R\to R^n\) be a right \(\sigma\)-derivation and let \(S\)
be a \((\sigma,\delta)\)-free skew extension of \(R\) 
generated by \(x_1,\dots,x_n\).
Given an \(R\)-ring \(A\) and \(a_1,\dots,a_n\) elements of
\(A\) such that
\begin{equation}\label{eq:upsfe}
ra_j=\sum_{i=1}^n a_i\sigma_{ij}(r) + \delta_j(r),
\quad\text{for all \(j=1,\dots,n\) and \(r\in R\),}
\end{equation}
there exists a unique \(R\)-ring homomorphism \(\varphi\colon S\to A\)
such that \(\varphi(x_i)=a_i\), for all \(i=1,\dots,n\).
\end{proposition}

\begin{proof}
Since the submonoid \(\langle X\rangle\) of \(S^{\times}\)
is free on \(X=\{x_1,\dots,x_n\}\), there is a monoid morphism
\(\rho\colon \langle X\rangle \to A^{\times}\) such that \(\rho(x_i)=a_i\),
for all \(i=1,\dots,n\). Let \(\varphi\colon S\to A\) be the
right \(R\)-module homomorphism such that \(\varphi(w)=\rho(w)\),
for all \(w\in\langle X\rangle\). Let us show that \(\varphi\) is
an \(R\)-ring homomorphism. First, note that if \(f=\sum_{w\in\langle X\rangle}
wr_w\) is an arbitrary element of \(S\), then, for \(j=1,\dots,n\), one has
\begin{equation*}\begin{split}
\varphi(fx_j) & = \varphi\biggl(\sum_{i=1}^n\sum_{w\in\langle X\rangle}wx_i\sigma_{ij}(r_w)
    + \sum_{w\in\langle X\rangle} w\delta_j(r_w)\biggr)\\
	& = \sum_{i=1}^n\sum_{w\in\langle X\rangle}\rho(wx_i)\sigma_{ij}(r_w) 
	  + \sum_{w\in\langle X\rangle}\rho(w)\delta_j(r_w)\\
	& = \sum_{w\in\langle X\rangle}\rho(w)\biggl(\sum_{i=1}^na_i\sigma_{ij}(r_w) + \delta_j(r_w)
	  \biggr)\\
	& = \sum_{w\in\langle X\rangle}\rho(w)r_wa_j\\
	& = \varphi(f)a_j.
\end{split}\end{equation*}
Hence, by an inductive argument on the length of a monomial \(w\), one proves that
\[
\varphi(fw)=\varphi(f)\rho(w),
\]
for all \(w\in\langle X\rangle\). Now, given \(f'=\sum_{w\in\langle X\rangle} wr_w'\in S\),
one has
\[
\varphi(ff')=\varphi\biggl(f\sum_{w\in\langle X\rangle}wr_w'\biggr)
=\sum_{w\in\langle X\rangle} \varphi(fw)r_w' = \sum_{w\in\langle X\rangle} \varphi(f)\rho(w)
r_w'=\varphi(f)\varphi(f').
\]
It follows that \(\varphi\) is an \(R\)-ring homomorphism such that \(\varphi(x_i)=a_i\),
for all \(i=1,\dots,n\). 

Uniqueness of \(\varphi\) follows from the fact that
\(S\) is generated by \(\{x_1,\dots,x_n\}\) as an \(R\)-ring.
\end{proof}

\begin{corollary}\label{cor:uniqsfe}
Let \(R\) be a ring, let \(n\) be a positive integer,
let \(\sigma\colon R\to M_n(R)\) be a ring homomorphism, let
\(\delta \colon R\to R^n\) be a right \(\sigma\)-derivation
and let \(S_1\)
be a \((\sigma,\delta)\)-free skew extension of \(R\) 
generated by \(x_1,\dots,x_n\).
Let \(S_2\) be an \(R\)-ring containing
elements \(y_1,\dots,y_n\) such that 
\begin{equation*}
ry_j=\sum_{i=1}^n y_i\sigma_{ij}(r) + \delta_j(r),
\quad\text{for all \(j=1,\dots,n\) and \(r\in R\),}
\end{equation*} and such that
for any \(R\)-ring \(A\) and
elements \(a_1,\dots,a_n\) satisfying \eqref{eq:upsfe}
there exists a unique \(R\)-ring homomorphism \(\psi\colon S_2\to A\) such that 
\(\psi(y_i)=a_i\), for all \(i=1,\dots,n\). Then there is an \(R\)-ring
isomorphism \(S_1\to S_2\), sending \(x_i\) to \(y_i\), for all \(i=1,\dots,n\).

In particular if \(S_1\) and \(S_2\) are \((\sigma,\delta)\)-free skew extensions of \(R\)
generated by \(x_1,\dots,x_n\) and \(y_1,\dots,y_n\), respectively, then 
there exists a unique \(R\)-ring isomorphism \(S_1\to S_2\) sending \(x_i\) to \(y_i\)
for all \(i=1,\dots,n\).
\end{corollary}

\begin{proof}
Proposition~\ref{prop:univsfe} provides an \(R\)-ring homomorphism \(\varphi\colon S_1\to S_2\)
such that \(\varphi(x_i)=y_i\), for all \(i=1,\dots,n\), and the universal property in
the hypothesis
gives an \(R\)-ring homomorphism \(\psi\colon S_2\to S_1\) such that
\(\psi(y_i)=x_i\), for all \(i=1,\dots,n\). Since \(\psi\varphi\) and \(\Id_{S_1}\) are 
both \(R\)-ring homomorphism \(S_1\to S_1\) sending \(x_i\) to \(x_i\), uniqueness in 
Proposition~\ref{prop:univsfe} gives \(\psi\varphi=\Id_{S_1}\). Similarly, uniqueness
in the hypothesis gives \(\varphi\psi=\Id_{S_2}\). So, \(\varphi\) is an \(R\)-ring isomorphism.
\end{proof} 

It follows that the universal property in Proposition~\ref{prop:univsfe} completely
characterizes skew free extensions. This fact will be used in the proof of
Corollary~\ref{cor:tr}.

In view of Theorem~\ref{th:existsfe} and Corollary~\ref{cor:uniqsfe}, given
a ring \(R\), a positive integer \(n\), a ring homomorphism \(\sigma\colon R\to M_n(R)\) 
and a right \(\sigma\)-derivation \(\delta\colon R\to R^n\), we can speak of the \((\sigma,\delta)\)-free skew extension 
of \(R\) generated by \(x_1,\dots,x_n\). This ring will, henceforth, be denoted by
\(R\langle x_1,\dots,x_n;\sigma,\delta\rangle\).

It is often the case that, as is the case of ordinary skew polynomial rings
(cf.~\cite[2.1.3]{pC1995}), a suitable change of variables allows one to reduce the
right \(\sigma\)-derivation \(\delta\) to the zero map or the ring homomorphism 
\(\sigma\) to the scalar map, as the next result shows.

\begin{proposition}\label{prop:inner}
Let \(R\) be a ring with center \(C\), let \(n\) be a positive integer, let 
\(\sigma\colon R\rightarrow M_n(R)\) be a ring homomorphism and let 
\(\delta\colon R\rightarrow R^n\) be a right \(\sigma\)-derivation. 
Let \(S=R\langle x_1,\dots,x_n;\sigma,\delta\rangle\)	be the \((\sigma,\delta)\)-free skew extension of \(R\) generated by \(x_1,\dots,x_n\).

\renewcommand{\labelenumi}{(\roman{enumi})}
\begin{enumerate}
	\item If there exists \(c\in C\) such that the matrix \(cI_n-\sigma(c)\) is 
	invertible in \(M_n(R)\), then, by a suitable choice of variables, 
	\(S=R\langle y_1,\ldots,y_n; \sigma, 0\rangle\), where \(0\) denotes the zero
	\(\sigma\)-derivation.
	\item Suppose that \(\sigma\) is diagonal, that \(\sigma(c)=cI_n\) for all \(c\in C\), 
	and that for every \(j=1,\dots, n\), there exists \(c_j\in C\) such that \(\delta_j(c_j)\) 
	is invertible in \(R\). Then, if \(\tau\colon R\to M_n(R)\) denotes the
	scalar homomorphism, there is a right \(\tau\)-derivation \(\delta'\) such that,
	by a suitable choice of variables, 
	\(S=R\langle y_1,\dots,y_n; \tau, \delta'\rangle\).
\end{enumerate}
\end{proposition}

\begin{proof}
(i) Consider the \(n\)-element set \(Y=\{y_1,\dotsc,y_n\}\) where \(y_j=x_jc-cx_j\),
for all \(j=1,\dots, n\). Note that
\[
[y_1 \ y_2 \ \dots \ y_n] =
[x_1 \ x_2 \ \dots \ x_n]\bigl(cI_n-\sigma(c)\bigr)-\delta(c).
\]
Then,
\[
[x_1 \ x_2 \ \dots \ x_n] = \bigl([y_1 \ y_2 \ \dots \ y_n]
+ \delta(c)\bigr)\bigl(cI_n-\sigma(c)\bigr)^{-1}.
\]	
It follows that that the submonoid \(\langle Y\rangle\) of \(S^{\times}\)
generated by \(Y\) is free on \(Y\) and that \(S\) is the free right \(R\)-module with 
basis \(\langle Y\rangle\). Moreover, for each \(j=1,\dots,n\) and \(r\in R\),
\begin{equation*}\begin{split}
ry_j & = r(x_jc-cx_j) \\
& = rx_jc-rcx_j \\
& = \sum_{i=1}^nx_i\sigma_{ij}(r)c +\delta_j(r)c - c\sum_{i=1}^nx_i\sigma_{ij}(r)-c\delta_j(r) \\
& =\sum_{i=1}^n(x_ic-cx_i)\sigma_{ij}(r) \\
& = \sum_{i=1}^n y_i\sigma_{ij}(r).
\end{split}\end{equation*}
This shows that \(S=R\langle y_1,\ldots,y_n; \sigma, 0\rangle\).

(ii) Since the elements \(c_j\) are in the center of \(R\),
\[
\delta_j(r)c_j+r\delta_j(c_j)=\delta_j(rc_j)=\delta_j(c_jr)=\delta_j(c_j)\sigma_{jj}(r)+c_j\delta_j(r),
\]
for all \(r\in R\) and \(j=1,\dots,n\). 
Hence, 
\[
\sigma_{jj}(r)\delta_j(c_j)^{-1}=\delta_j(c_j)^{-1}r,
\]
for all \(r\in R\) and \(j=1,\dots,n\).

For each \(j=1,\dots, n\), set \(y_j=x_j\delta_j(c_j)^{-1}\). For all \(r\in R\) and 
\(j=1,\dots,n\),
\begin{equation*}\begin{split}
	ry_j & = rx_j\delta_j(c_j)^{-1}\\
	& = \bigl(x_j\sigma_{jj}(r)+\delta_j(r)\bigr)\delta_j(c_j)^{-1} \\
	& = x_j\delta_j(c_j)^{-1}r+\delta_j(r)\delta_j(c_j)^{-1}\\
	& = y_jr+\delta_j(r)\delta_j(c_j)^{-1}
\end{split}\end{equation*} 
It remains to show that \(\delta_j'\colon R\rightarrow R\), defined as 
\(\delta_j'(r)=\delta_j(r)\delta_j(c_j)^{-1}\) for all \(r\in R\),
is an ordinary derivation. Indeed, let $r,s\in R$, then
\begin{equation*}\begin{split}
	\delta_j'(rs) & = \delta_j(rs)\delta_j(c_j)^{-1} \\
	& = \bigl(\delta_j(r)\sigma_{jj}(s)+r\delta_j(s)\bigr)\delta_j(c_j)^{-1}\\
	& = \delta_j(r)\delta_j(c_j)^{-1}s + r\delta_j(s)\delta_j(c_j)^{-1} \\
	& = \delta_j'(r)s+r\delta_j'(s).
\end{split}\end{equation*}
This proves that \(S=R\langle y_1,\dots,y_n; \tau, \delta'\rangle\).
\end{proof}

At this point it should be noted that when \(R\) is a division ring and
\(n\geq 2\), the center \(Z(S)\) of the free skew extension \(S=R\langle x_1,\dots,x_n;
\sigma,\delta\rangle\), as described in \cite[Proposition~3.1]{LM2024}, is
given by 
\[
Z(S) = \{r\in Z(R) \mid \sigma(r)=rI_n \text{ and }\delta(r)=0\}.
\]

\begin{remark}
One can consider skew free extensions with
coefficients on the left. For that, given a ring \(R\), a positive integer \(n\)
and a ring homomorphism \(\sigma\colon R\to M_n(R)\), a \emph{left \(\sigma\)-derivation} is
an additive map \(\delta\colon R\to {}^nR\), where \({}^nR\) stands for the (free)
right \(R\)-module formed by column vectors, such that
\[
\delta(rs) = \sigma(r)\delta(s) + \delta(r)s,
\]
for all \(r,s\in R\). The \((\sigma,\delta)\)-free skew extension of \(R\) generated by
\(x_1,\dots,x_n\) is then
defined to be the \(R\)-ring \(S\) which is free as left \(R\)-module on the
free monoid \(\langle X\rangle\) freely generated by \(X=\{x_1,\dots,x_n\}\)
satisfying
\(x_ir = \sum_{j=1}^n\sigma_{ij}(r)x_j + \delta_i(r)\),
for all \(i=1,\dots,n\) and \(r\in R\).
\end{remark}

\section{Skew free extensions as universal rings on bimodules}\label{sec:tr}

A skew free extension can also be obtained as a universal \(R\)-ring on an appropriate bimodule,
as we show in Corollary~\ref{cor:tr}, which closes this section. The advantage of
this approach is that it can be easily generalized.

\medskip

Given a ring \(R\) and an \(R\)-bimodule \(M\), the \emph{tensor ring on \(M\)} is defined to
be the (external) direct sum 
\[
T(M) = \bigoplus_{n\geq 0} M^{\otimes n},
\]
where \(M^{\otimes 0} = R\), \(M^{\otimes 1} = M\), and, for \(n\geq 2\),
 \(M^{\otimes n}=M\otimes_R\dots\otimes_R M\) (\(n\) factors).
The \(R\)-bimodule \(T(M)\) becomes a ring with multiplication given by
\[
(a_0,a_1,a_2,\dots)(b_0,b_1,b_2,\dots)
=(c_0, c_1, c_2, \dots),
\]
with \(c_n= \sum_{i=0}^n \rho_{i,n-i}(a_i\otimes b_{n-i})\), where \(\rho_{i,n-i}\colon
M^{\otimes i}\otimes_R M^{\otimes (n-i)}\to M^{\otimes n}\) are the natural \(R\)-bimodule maps. In fact, \(T(M)\)
is an \(R\)-ring with structure map given by
\begin{equation*}
\begin{array}{rcl}
\lambda\colon R & \longrightarrow & T(M)\\ r&\longmapsto &(r,0,0,\dots).
\end{array}
\end{equation*}
So \((T(M),\lambda)\) is an \(R\)-ring with an \(R\)-bimodule map
\begin{equation}\label{eq:bimodmap}
\begin{array}{rcl}
\mu\colon M & \longrightarrow & T(M)\\ x&\longmapsto &(0,x,0,\dots),
\end{array}
\end{equation}
which is universal with regard to this property, as shown in the next result.

\begin{proposition}\label{prop:tr1}
Let \(R\) be a ring, let \(M\) be an \(R\)-bimodule and let \((A,\lambda_A)\) be
an \(R\)-ring given with an \(R\)-bimodule map \(f\colon M\to A\). Then there
exists a unique \(R\)-ring homomorphism \(\psi\colon T(M)\to A\) such that
\(\psi \mu=f\). \qed
\end{proposition}

\begin{proof}
Define \(\psi_0=\lambda_A\), \(\psi_1=f\) and, for \(n\geq 2\), \(\psi_n\colon 
M^{\otimes n}\to A\)
to be the \(R\)-bimodule map induced by the \(n\)-balanced map
\[
\begin{array}{rcl}
M\times\dots\times M & \longrightarrow & A\\ 
(x_1,\dots,x_n)&\longmapsto &f(x_1)\dotsm f(x_n).
\end{array}
\]
Then \(\psi = \bigoplus_{n\geq 0} \psi_n\) is an \(R\)-ring homomorphism with
the desired property. Uniqueness follows from the fact that \(T(M)\) is generated by
\(\mu(M)\) as an \(R\)-ring.
\end{proof}

We shall need a related concept.

\begin{theorem}\label{th:2}
Let \(R\) be a ring, let \(M\) be an \(R\)-bimodule and let \(\xi\colon R\to M\) be an
\(R\)-bimodule map. Then, there exists an \(R\)-ring \((U,\lambda_U)\) and an \(R\)-bimodule
map \(j\colon M\to U\) such that \(j\xi=\lambda_U\) and which satisfies the following universal property: given any
\(R\)-ring \((A,\lambda_A)\) and any \(R\)-bimodule map \(f\colon M\to A\) such that
\(f\xi=\lambda_A\), there exists a unique \(R\)-ring homomorphism \(\varphi\colon U\to A\)
such that \(\varphi j=f\).

Moreover, if \((U',\lambda_{U'})\) is an \(R\)-ring given with an \(R\)-bimodule map
\(j'\colon M\to U'\) such that \(j'\xi=\lambda_{U'}\) satisfying the universal property
satisfied by \(U\), then there exists a unique \(R\)-ring isomorphism \(\psi\colon U\to U'\)
such that \(\psi j = j'\). 
\end{theorem}

\begin{proof}
Let \((T(M),\lambda)\) be the tensor ring on \(M\), let \(\mu\) be the bimodule
map defined in \eqref{eq:bimodmap}, let \(U=T(M)/I\), 
where \(I\) is the two-sided 
ideal of \(T(M)\) generated by \(\Imag(\lambda-\mu\xi)\) and let \(\pi\colon T(M)\to U\)
be the canonical surjective ring homomorphism. Then, \(U\) inherits the \(R\)-ring
structure from \(T(M)\) and, letting \(j=\pi\mu\), clearly, \(j\xi=\lambda_U\).

Finally, let \((A,\lambda_A)\) be an \(R\)-ring and let \(f\colon M\to A\) be an
\(R\)-bimodule map such that \(f\xi=\lambda_A\). Let \(\psi\colon T(M)\to A\) the
\(R\)-ring homomorphism satisfying \(\psi\mu=f\) afforded by Proposition~\ref{prop:tr1}.
Then \(\psi(\lambda-\mu\xi)=\psi\lambda-f\xi=\lambda_A-\lambda_A=0\). Therefore,
\(I\subseteq\ker\psi\), and, so, there exists a ring homomorphism \(\varphi\colon U\to A\)
such that \(\varphi\pi=\psi\). So, \(\varphi j =\varphi\pi\mu=\psi\mu=f\). Again,
uniqueness follows from the fact that, as an \(R\)-ring, \(U\) is generated by the image
of \(j\).

The last statement follows by applying the universal property twice for the data
\((U,\lambda_U,j)\) and \((U',\lambda_{U'},j')\) and comparing the composition of the
\(R\)-ring homomorphisms thus obtained with the respective identity maps.
\end{proof}

The \(R\)-ring \(U=U(M,\xi)\) constructed above, whose uniqueness up to an \(R\)-ring
isomorphism is guaranteed by the last statement of Theorem~\ref{th:2},
will be called the \emph{universal \(R\)-ring on the
bimodule \(M\) along \(\xi\)}. 


A particular case is of interest.
What follows is an elaboration on the comments at the bottom of p.~113 of \cite{pC1985}.

Given a right \(R\)-module \(V\), we shall denote the ring of all \(R\)-module endomorphisms 
of \(V\) by \(\End(V_R)\), with multiplication given by composition of functions. The set 
\(\Hom(V_R,R_R)\) of all \(R\)-module maps \(f\colon V\to R\) has a natural left \(R\)-module
structure such that \((rf)(v)=rf(v)\), for all \(r\in R\) and \(v\in V\). Given a ring
homomorphism \(\sigma\colon R\to \End(V_R)\), we say that a map \(\delta\colon R\to
\Hom(V_R,R_R)\) is a \emph{\(\sigma\)-derivation} if it is additive and \(\delta(rs)
=\delta(r)\sigma(s)+r\delta(s)\), for all \(r,s\in R\).

Suppose, now, that we are given a right \(R\)-module \(M\) having \(R\) as a complemented
\(R\)-submodule, that is, suppose that \(R\) is a right \(R\)-submodule of \(M\) and
that there exists a right \(R\)-submodule \(V\) of \(M\) such that \(M\) is the
(internal) direct sum \(M=R\oplus V\). Suppose, further, that \(M\) has a left \(R\)-module
structure making it into an \(R\)-bimodule in such a way that \(R\) is a sub-bimodule. 
This implies that there exist functions \(\sigma\colon R\to \End(V_R)\) and
\(\delta\colon R\to \Hom(V_R,R_R)\) such that
\begin{equation}\label{eq:ds}
r(s+v) = rs + \delta(r)(v) + \sigma(r)(v), \quad\text{for all \(r,s\in R\) and \(v\in V\).}
\end{equation}
It is then a straightforward verification that \(\sigma\) must be a ring homomorphism and
\(\delta\) a \(\sigma\)-derivation. Conversely, given a right \(R\)-module \(V\),
a ring homomorphism \(\sigma\colon R\to\End(V_R)\) and a \(\sigma\)-derivation \(\delta
\colon R\to \Hom(V_R,R_R)\), \eqref{eq:ds} defines an \(R\)-bimodule structure on the
right \(R\)-module \(M=R\oplus V\) in which \(R\) is a sub-bimodule. We shall call \(M\) the
\emph{rank \(1\) bimodule extension of \(V\) defined by \(\sigma\) and \(\delta\)}. Note
that when \(\delta\) is the zero derivation, \(V\) is a sub-bimodule of \(M\) and \(R\)
is, then, complemented as a sub-bimodule. In this case, the universal \(R\)-ring on \(M\)
coincides with the tensor ring on \(V\), as the following result shows.

\begin{proposition}
Let \(R\) be a ring, let \(V\) be a right \(R\)-module, let \(\sigma\colon R\to
\End(V_R)\) be a ring homomorphism and let \(M\) be the rank \(1\) bimodule extension of \(V\)
defined by \(\sigma\) and the zero derivation. Then there exists an \(R\)-ring
isomorphism \(U(M,\iota)\cong T(V)\), where \(\iota\colon R\to M\) is the inclusion map.
\end{proposition}

\begin{proof}
For an \(R\)-ring homomorphism \(U(M,\iota)\to T(V)\) use the universal property of \(U(M,\iota)\) with 
bimodule map \(f\colon M\to T(V)\) given by \(f(s+v)=(s,v,0,0,\dots)\), for all
\(s\in R\) and \(v\in V\). Its inverse is the \(R\)-ring homomorphism \(T(V)\to U(M,\iota)\)
induced by the restriction to $V$ of the bimodule map \(j\colon M\to U(M,\iota)\) in the
statement of Theorem~\ref{th:2}.
\end{proof}

Now, we return to skew free extensions, as defined in Section~\ref{sec:def}.
Given a ring \(R\), a positive integer \(n\), a ring homomorphism \(\sigma\colon R\to M_n(R)\) 
and a right \(\sigma\)-derivation \(\delta\colon R\to R^n\), 
let \(S=R\langle x_1,\dots,x_n;\sigma,\delta\rangle\)
be the \((\sigma,\delta)\)-free skew extension of \(R\) generated by \(x_1,\dots,x_n\).
We shall show that \(S\) can be regarded as a
universal \(R\)-ring on a rank \(1\) bimodule extension of a free right \(R\)-module
of rank \(n\). 

So let \(V\) be the free right \(R\)-module of rank \(n\) with basis \(v_1,\dots,v_n\).
Let \(\overline{\sigma} \colon R\to \End(V_R)\) be the ring homomorphism obtained by
composing \(\sigma\) with the ring isomorphism \(M_n(R)\to\End(V_R)\) which
sends the matrix \((r_{ij})\) to the endomorphism of \(V\) such that 
\(v_j\mapsto \sum_{i=1}^{n}v_ir_{ij}\). And let \(\overline{\delta}\colon R\to \Hom(V_R,R_R)\)
be obtained by composing \(\delta\) with the left \(R\)-module isomorphism \(R^n\to
\Hom(V_R,R_R)\) which sends \([r_1 \ r_2 \ \dots \ r_n]\) to the right \(R\)-module
map from \(V\) to \(R\) such that \(v_i\mapsto r_i\). It is straightforward to verify that 
\(\overline{\delta}\) is a \(\overline{\sigma}\)-derivation. Let \(M=R\oplus V\) be the
rank \(1\) bimodule extension of \(V\) defined by \(\overline{\sigma}\) 
and \(\overline{\delta}\), and let \(U(M,\iota)\) be the universal \(R\)-ring on \(M\)
along the inclusion map \(\iota\colon R\to M\).

Our task is to show that \(U(M,\iota)\) and \(S\) are isomorphic \(R\)-rings. We accomplish this by showing
that \(U(M,\iota)\) satisfies the universal property of Proposition~\ref{prop:univsfe}
that characterizes \(S\).

First, note that in \(U(M,\iota)\) one has 
\[
\lambda_U(r)\overline{v_j}=\sum_{i=1}^n\overline{v_i}\lambda_U(\sigma_{ij}(r)) + \lambda_U(\delta_j(r)),
\]
for all \(j=1,\dots, n\), where \(\overline{v_j}=\pi\mu(v_j)\), in the notation of the
proof of Theorem~\ref{th:2}. Now, let \(A\) be an \(R\)-ring and let \(a_1,\dots,a_n\) be elements of \(A\) satisfying \eqref{eq:upsfe}. The right \(R\)-module map 
\(f\colon M\to A\) such that \(1\mapsto 1_A\) and \(v_i\mapsto a_i\), for all 
\(i=1,\dots, n\) is, in fact, an \(R\)-bimodule map and, clearly, \(f\iota=\lambda_A\),
the ring homomorphism \(R\to A\) affording \(A\) its \(R\)-ring structure. 
The universal property of \(U(M,\iota)\) guarantees that there exists a unique
\(R\)-ring homomorphism \(\varphi\colon U(M,\iota)\to A\) such that 
\(\varphi(\overline{v_i})=f(v_i)=a_i\), for all \(i=1,\dots,n\). It follows that
\(U(M,\iota)\cong 
S\) as \(R\)-rings. We have proved

\begin{corollary}\label{cor:tr}
Given a ring \(R\), a positive integer \(n\), a ring homomorphism \(\sigma\colon R\to M_n(R)\) 
and a right \(\sigma\)-derivation, the \((\sigma,\delta)\)-free skew extension 
\(R\langle x_1,\dots,x_n;\sigma,\delta\rangle\)
is isomorphic, as an \(R\)-ring, to the universal \(R\)-ring \(U(M,\iota)\), where
\(M\) is a rank \(1\) bimodule extension of a free right \(R\)-module of rank \(n\)
and \(\iota\colon R\to M\) is the inclusion map.
\qed
\end{corollary}

We remark that this perspective allows one to consider skew free extensions on
a not necessarily finite set of generators, by taking universal \(R\)-rings of
the form \(U(M,\iota)\), where \(M\) is a rank \(1\) bimodule extension of an
arbitrary free right \(R\)-module. Ring homomorphisms into 
column-finite infinite matrices would be an alternative approach.

\section{Regularity}\label{sec:domain}

%

In this section, we shall investigate conditions for a skew free extension to be a domain.

\medskip

As is well-known, given a ring \(R\), a ring endomorphism \(\sigma\colon R\to R\)
and a right \(\sigma\)-derivation \(\delta\colon R\to R\), the skew
polynomial ring \(R[x;\sigma,\delta]\) is a domain provided that \(R\) is
a domain and \(\sigma\) is injective (cf., \textit{e.g.}, \cite[2.9]{MR2001}). 
Once more than one variable is present, these conditions are no longer sufficient,
as the following simple example shows.

\begin{example}
Let \(k\) be a field, let \(\sigma\colon k[t]\to M_2(k[t])\) be the
\(k\)-algebra homomorphism such that 
\[
\sigma(t)=\begin{pmatrix} t&0\\0&0\end{pmatrix},
\]
that is, if \(f=\lambda_0+\lambda_1t+\dots+\lambda_nt^n\in k[t]\), with 
\(\lambda_i\in k\), for all \(i=0,\dots,n\), then
\[
\sigma(f)
=\begin{pmatrix} f&0\\0&\lambda_0\end{pmatrix},
\]
and let \(\delta\) be the zero right \(\sigma\)-derivation.
Consider the \((\sigma,\delta)\)-skew free extension \(S=k[t]\langle x_1,x_2;\sigma,\delta\rangle\)
of \(k[t]\). Although \(k[t]\) is a domain and \(\sigma\) is injective, in \(S\) one has
\[
tx_2 = x_1\sigma_{12}(t) + x_2\sigma_{22}(t) = 0,
\]
but both \(x_2\) and \(t\) are nonzero.
\end{example}

We now introduce the key concept of this section.

Given a ring \(R\), a positive integer \(n\) and a ring homomorphism \(\sigma\colon
R\to M_n(R)\), define a sequence of ring homomorphisms \(\sigma^{(r)}\colon
R\to M_{n^r}(R)\) inductively by \(\sigma^{(1)}=\sigma\) and for \(a\in R\), if
\[
\sigma^{(r)}(a) = \begin{pmatrix}
a_{11} &\dots& a_{1n^r} \\
\vdots & \ddots & \vdots \\
a_{n^r 1} & \dots & a_{n^rn^r}
\end{pmatrix}\in M_{n^r}(R),
\]
then
\[
\sigma^{(r+1)}(a) = \begin{pmatrix}
\sigma(a_{11}) &\dots& \sigma(a_{1n^r}) \\
\vdots & \ddots & \vdots \\
\sigma(a_{n^r 1}) & \dots & \sigma(a_{n^rn^r})
\end{pmatrix}\in M_{n^{r+1}}(R).
\]

\begin{definition}
We say that ring homomorphism \(\sigma\colon R\to M_n(R)\) is \emph{megainjective}
if for all \(a\in R\setminus\{0\}\) and \(r\geq 1\), \(\sigma^{(r)}(a)\) is neither the
zero matrix nor a left 
zero divisor  
in the ring \(M_{n^r}(R)\). 
 
\end{definition}

Equivalently, \(\sigma\) is megainjective if the columns of \(\sigma^{(r)}(a)\) are right \(R\)-linearly independent for all \(r\geq 1\) and \(a\in R\setminus\{0\}\). Clearly, if \(\sigma\) is megainjective, then \(\sigma\) is injective and \(R\) is a domain. In fact, for \(n=1\), a
ring homomorphism \(\sigma\colon R\to R\) is megainjective if, and only if, \(R\) is
a domain and \(\sigma\) is injective.

We shall illustrate this concept with a few examples.

\begin{example}\label{ex:dr} If \(R\) is a division ring, then any ring homomorphism
\(\sigma\colon R\to M_n(R)\) is megainjective because, in this case, \(\sigma(a)\), and therefore 
\(\sigma^{(r)}(a)\), is invertible for all \(a\in R\setminus\{0\}\) and \(r\geq 1\).
\end{example}

\begin{example}\label{ex:ut}
If \(R\) is a domain and \(\sigma\colon R\to M_n(R)\) is a ring homomorphism such
that \(\sigma(a)\) is upper (respectively, lower) triangular with \(\sigma_{ii}(a)\neq 0\)
for all \(a\in R\setminus\{0\}\) and \(i=1,\dots,n\), then \(\sigma\) is megainjective, because,
in this case, for \(r\geq 1\), the matrix 
\(\sigma^{(r)}(a)\) is upper (respectively, lower) triangular with all the elements on the 
main diagonal being nonzero, but this implies that 
the columns of $\sigma^{(r)}(a)$ are right linearly independent over \(R\).
\end{example}

We say that a ring homomorphism \(\sigma\colon R\to M_n(R)\) is \emph{upper (respectively, lower) 
triangularizable} if there exists an invertible matrix \(P\in M_n(R)\) such that
\(P\sigma(a)P^{-1}\) is upper (respectively, lower) triangular for all \(a\in R\). Therefore,
if \(R\) is a domain and \(\sigma\colon R\to M_n(R)\) is
upper triangularizable under conjugation by \(P\) with all diagonal entries of 
\(P\sigma(a)P^{-1}\) nonzero for all \(a\in R\setminus\{0\}\), then \(\sigma\) is
megainjective. This is because the ring homomorphism \(\tau\colon R\to M_n(R)\)
defined by \(\tau(a)=P\sigma(a)P^{-1}\), for \(a\in R\), satisfies the condition
in Example~\ref{ex:ut} for being megainjective and, for all \(a\in R\), the
matrices \(\tau^{(r)}(a)\) and \(\sigma^{(r)}(a)\) are conjugate over \(R\).

\begin{example}
If \(R\) is a right Ore domain and \(\sigma\colon R\to M_n(R)\) is a ring 
homomorphism such that, for all
\(a\in R\setminus\{0\}\), the columns of \(\sigma(a)\) are right \(R\)-linearly independent, 
then \(\sigma\) is megainjective. This fact can be proved as follows. Suppose the Ore division ring of fractions is \(\varphi\colon R\hookrightarrow Q\). Given
a positive integer \(s\), and a ring homomorphism \(\psi\colon S\to S'\), we shall denote 
by \(M_s(\psi)\) the homomorphism from \(M_s(S)\) into \(M_s(S')\) defined
by \(M_s(\psi)\bigl((s_{ij})\bigr)=\bigl(\psi(s_{ij})\bigr)\). 
If \(A\in M_n(R)\) is a matrix whose columns are right \(R\)-linearly
independent, then, clearly, the columns of \(M_n(\varphi)(A)\) are right \(Q\)-linearly independent, and,
therefore, \(M_n(\varphi)(A)\) is invertible in \(M_n(Q)\). Thus,
there exists a homomorphism \(\overline{\sigma}\colon Q\rightarrow M_n(Q)\)
such that the  diagram 
\[
\begin{tikzcd}
	R &&& {M_n(R)} \\
	\\
	Q &&& {M_n(Q)}
	\arrow["\sigma", from=1-1, to=1-4]
	\arrow["\varphi", from=1-1, to=3-1]
	\arrow["{M_n(\varphi)}", from=1-4, to=3-4]
	\arrow["{\overline{\sigma}}", from=3-1, to=3-4]
\end{tikzcd}
\]
commutes. Applying the functor \(M_n\) to this diagram, we obtain
the commutativity of the square on the right hand-side of
\[
\begin{tikzcd}
	R &&& {M_n(R)} &&& {M_{n^2}(R)}\\
	\\
	Q &&& {M_n(Q)} &&& {M_{n^2}(Q)}
	\arrow["\sigma", from=1-1, to=1-4]
	\arrow["\varphi", from=1-1, to=3-1]
	\arrow["{M_n(\varphi)}", from=1-4, to=3-4]
	\arrow["{\overline{\sigma}}", from=3-1, to=3-4]
	\arrow["{M_n(\sigma)}", from=1-4, to=1-7]
	\arrow["{M_{n^2}(\varphi)}", from=1-7, to=3-7]
	\arrow["{M_n(\overline{\sigma})}", from=3-4, to=3-7]
\end{tikzcd}
\]

If \(a\in R\setminus\{0\}\) is such that the columns of \(\sigma(a)\) are right \(R\)-linearly
independent, then, as we have seen above, \(\overline{\sigma}\varphi(a)\) is invertible 
in \(M_n(Q)\) and,
hence, \(M_{n^2}(\varphi)M_n(\sigma)\sigma(a)=
M_{n^2}(\overline{\sigma})\bigl(\overline{\sigma}\varphi(a)\bigr)\) is an invertible
element of \(M_{n^2}(Q)\). Since \(M_n(\sigma)\sigma=\sigma^{(2)}\),  this implies that the 
columns of \(\sigma^{(2)}(a)\) must be right \(R\)-linearly independent. An inductive argument
proves that the columns of \(\sigma^{(r)}(a)\) are right \(R\)-linearly independent for all
\(r\geq 1\).
\end{example}

Now let \(\sigma\colon R\to M_n(R)\) be a ring homomorphism, let \(\delta\colon R\to R^n\)
be a right \(\sigma\)-derivation, and consider the \((\sigma,\delta)\)-free skew extension 
\(S=R\langle x_1,\dots,x_n;\sigma,\delta\rangle\) of \(R\). So \(S\) is a ring which contains 
\(R\), as a subring, and an \(n\)-element set \(X=\{x_1,\dots,x_n\}\) freely generating a submonoid
\(\langle X\rangle\) of its multiplicative monoid. Moreover, the elements of
\(S\) are 
written uniquely in the form
\[
\sum_{w\in\langle X\rangle} wa_w,
\]
where, for each \(w\in \langle X \rangle\), \(a_w\in R\) and the set \(\{w\in \langle X \rangle
\mid r_w\neq 0\}\) is finite, and, in \(S\), the
commutation rules 
\begin{equation}\label{x_ia}
ax_j = \sum_{i=1}^n x_i\sigma_{ij}(a) + \delta_j(a),
\quad\text{for all \(j=1,\dots,n\) and \(a\in R\),}
\end{equation}
are satisfied.

We shall see, in Corollary~\ref{cor:sfedom}, that a sufficient condition for 
\(S\) to be a domain is that 
\(\sigma\) is megainjective. For that, we shall need a lemma, which
uses an inductive argument based on a lexicographic ordering on \(\langle X\rangle\) 
described below.

For a positive integer \(r\), denote the subset of \(\langle X\rangle\) consisting of 
all monomials of length \(r\) by \(\langle X\rangle_r\). Order the free generators in the natural way, that is, 
declare \(x_i<x_j\) if, and only if, \(i<j\), and then order \(\langle X\rangle_r\)
lexicographically: \(x_{i_1}x_{i_2}\dotsm x_{i_r}
< x_{j_1}x_{j_2}\dotsm x_{j_r}\) if, and only if, \(x_{i_1}<x_{j_1}\) or 
there exists \(s\in\{2,\dots,r\}\) such that
\(x_{i_1}=x_{j_1}, x_{i_2}=x_{j_2},\dots, x_{i_{s-1}}=x_{j_{s-1}}\) and \(x_{i_s}<x_{j_s}\).
This defines a total ordering on the set \(\langle X\rangle_r\).
We shall list the \(n^r\) elements of \(\langle X\rangle_r\) ascendantly according to 
this order:
\begin{equation}\label{eq:orderingwords}
w_1<w_2<\dots <w_{n^r},
\end{equation}
and do the same with the \(n^{r-1}\) elements of \(\langle X\rangle_{r-1}\) 
lexicographically ordered:
\begin{equation}\label{eq:orderingwordsv}
v_1<v_2<\dots<v_{n^{r-1}}.
\end{equation}


Given \(k,l\in\{1,\dots,n^{r-1}\}\) and \(p,q\in\{1,\dots,n\}\), it easy to see that
\[
v_kx_p < v_lx_q\quad\text{if, and only if, \(k<l\) or (\(k=l\) and \(p<q\)).}
\]
Therefore, there are exactly \(n(l-1)+q-1\) elements in \(\langle X\rangle_r\)
that are smaller than \(v_lx_q\). It follows that \(v_lx_q=w_{n(l-1)+q}\). Or,
equivalently, given \(j\in\{1,2,\dots,n^r\}\), let \(l\) and \(q\) be the integers such 
that \(j=n(l-1)+q\)
with \(1\leq q \leq n\), then \(w_j=v_lx_q\).

For a non-negative integer \(m\), we shall denote the right \(R\)-submodule of \(S\)
generated by all the monomials of \(\langle X\rangle\) of 
length \(\leq m\) by \(S_m\). 
It follows from \eqref{x_ia} that 
\begin{equation}\label{eq:xism}
S_mx_j\subseteq S_{m+1}, \quad\text{for all \(j=1,\dots,n\) and \(m\geq 0\).}
\end{equation}

\begin{lemma}\label{lem:scalartimesword}
Let \(r\) be a positive integer, let
\(j\in\{1,2,\dots, n^r\}\), let \(w_j\in\langle X\rangle_r\) be as in 
\eqref{eq:orderingwords} and let \(a\in R\). Then
\begin{equation*}
    aw_j \equiv \sum_{i=1}^{n^r} w_i\sigma^{(r)}_{ij}(a) \ \dpmod{S_{r-1}},
\end{equation*}
where, for all \(i=1,\dots, n^r\), \(w_i\in \langle X\rangle_r\) are listed according
to \eqref{eq:orderingwords}.
\end{lemma}

\begin{proof}
    We proceed by induction on \(r\), with the base case \(r=1\) following from \eqref{x_ia}. 
		
Suppose that \(r\geq 2\) and that the result holds for \(r-1\). By what we have seen above, 
we can write 
\(w_j=v_lx_q\), for some \(q\in\{1,\dots,n\}\), \(l\in\{1,\dots,n^{r-1}\}\) and
\(v_l\in \langle X\rangle_{r-1}\) as in \eqref{eq:orderingwordsv}. 
Using the induction hypothesis, there exists \(f\in S_{r-2}\) such that 
    \begin{equation*}\begin{split}
    aw_j &= av_lx_q\\ 
    &=\biggl(\sum_{k=1}^{n^{r-1}}v_k\sigma^{(r-1)}_{kl}(a)+f\biggr)x_q\\
    &=\sum_{k=1}^{n^{r-1}}\sum_{i=1}^nv_kx_i\sigma_{iq}\bigl(\sigma^{(r-1)}_{kl}(a)\bigr) +
    \sum_{k=1}^{n^{r-1}}v_q\delta_q\bigl(\sigma^{(r-1)}_{kl}(a)\bigr) +fx_q.
    \end{split}\end{equation*}
		Since \(\sigma_{iq}\bigl(\sigma^{(r-1)}_{kl}(a)\bigr)\) are precisely the elements in
		column \(n(l-1)+q\) of \(\sigma^{(r)}(a)\) and \(\sum_{k=1}^{n^{r-1}}v_q
		\delta_q\bigl(\sigma^{(r-1)}_{kl}(a)\bigr) +fx_q\in S_{r-1}\), the result follows.
\end{proof}

As a consequence of \eqref{eq:xism}, we have that \(S_mw\subseteq S_{m+\lvert w \rvert}\),
for all \(w\in\langle X\rangle\) and \(m\geq 0\), and, so,
\(S_m S_n\subseteq S_{m+n}\), 
for all \(m,n\geq 0\).
Because, \(S_0\subseteq  S_1\subseteq S_2\subseteq \dots\) and \(\cup_{m\geq 0} S_m = S\), the function \(\deg\colon S \to \mathds{N}\cup
\{-\infty\}\) defined by
\begin{equation}\label{eq:ps}
\deg(f) = \min\{m \mid f\in S_m\},
\end{equation}
where we set \(S_{-\infty}=\{0\}\), is a \emph{pseudo-valuation} as defined in
\cite[Sec.~2.2]{pC1985}, that is, one has
\renewcommand{\labelenumi}{\Roman{enumi}.}
\begin{enumerate}
	\item \(\deg(f)\geq 0\), for all \(f\in S\setminus\{0\}\), and \(\deg(0)=-\infty\),
	\item \(\deg(f-g)\leq\max\{\deg(f),\deg(g)\}\), for all \(f,g\in S\),
	\item \(\deg(fg)\leq\deg(f)+\deg(g)\), for all \(f,g\in S\),
	\item \(\deg(1)=0\).
\end{enumerate}
When equality hols in III, we say that \(\deg\) is a \emph{degree-function}.

In the next result, we present necessary and sufficient conditions for \(\deg\) to
be a degree-function.

\begin{theorem}\label{theo:degreeisadditive}
    Let \(R\) be a ring, let \(n\) be a positive integer, let \(\sigma\colon R\to M_n(R)\) be a ring homomorphism, let \(\delta\colon R\to R^n\) be a right \(\sigma\)-derivation and let \(\deg\) be the
		pseudo-valuation on 
		the \((\sigma,\delta)\)-free skew extension \(S=R\langle x_1,\dots,x_n;\sigma,\delta
		\rangle\)
of \(R\), as defined in \eqref{eq:ps}. Then \(\deg\) is a degree-function
if, and only if, \(\sigma\) is megainjective.
\end{theorem}

\begin{proof}
Fix \(a\in R\), let \(r\) be a positive integer, let 
\(L_1, \dots, L_{n^r}\) be the columns of \(\sigma^{(r)}(a)\) and let
\(w_1<w_2<\dots<w_{n^r}\) be the monomials of \(\langle X\rangle_r\). Given 
\(b_1,b_2,\dots,b_{n^r}\in R\), by Lemma~\ref{lem:scalartimesword}, 
\[
a\sum_{j=1}^{n^r}w_jb_j \equiv
\sum_{i=1}^{n^r}w_i\sum_{j=1}^{n^r}\sigma^{(r)}_{ij}(a)b_j \ \dpmod{S_{r-1}}.
\]
Therefore, 
\begin{equation}\label{eq:equivalence}
\deg\biggl(a\sum_{j=1}^{n^r}w_jb_j\biggl) <r,
\quad\text{if, and only if,}\quad
L_1b_1+L_2b_2+\dots+L_{n^r}b_{n^r}=0, 
\end{equation}
since the elements of \(\langle X\rangle_r\)
are right linearly independent over \(R\).

Suppose that \(\deg\) is a degree function. 
If \(a\neq 0\) and \(b_1,b_2,\dots,b_{n^r}\in R\) are such that \(L_1b_1+L_2b_2+\dots+
L_{n^r}b_{n^r}=0\), then, by \eqref{eq:equivalence}, 
\[
r>\deg\biggl(a\sum_{j=1}^{n^r}w_jb_j\biggl) = 
\deg(a)+\deg\biggl(\sum_{j=1}^{n^r}w_jb_j\biggr)=
\deg\biggl(\sum_{j=1}^{n^r}w_jb_j\biggr),
\]
which is only possible if \(b_1=b_2=\dots=b_{n^r}=0\). 
It follows that \(\sigma\) is megainjective.

Conversely, suppose that \(\sigma\) is megainjective. Let
\(f,g\in S\). If \(f=0\) or \(g=0\), then, clearly, \(\deg(fg)=\deg(f)+\deg(g)\). Assume
that both \(f\) and \(g\) are nonzero, with, say, \(\deg(f)=s, \deg(g)=r\) and write
\[
f = \sum_{l=1}^{n^s}v_la_l + \overline{f} \quad\text{and}\quad
g = \sum_{j=1}^{n^r}w_jb_j + \overline{g}
\]
with \(a_l\in R\), not all zero, \(b_j\in R\), not all zero,
\(\{v_1,\dots,v_{n^s}\}=\langle X\rangle_s\),
\(\{w_1,\dots,w_{n^r}\}=\langle X\rangle_r\), \(\overline{f}\in S_{s-1}\),
\(\overline{g}\in S_{r-1}\).
Then,
\begin{equation}\label{eq:fg}
fg = \sum_{l=1}^{n^s}v_l\Bigl(a_l\sum_{j=1}^{n^r}w_jb_j\Bigr) + h,
\end{equation}
for some \(h\in S_{s+r-1}\).
Let \(l_0\) be such that \(a_{l_0}\neq 0\).
Because \(\sigma\) is
megainjective, the columns of \(\sigma^{(r)}(a_{l_0})\)
are right \(R\)-linearly independent, so, by \eqref{eq:equivalence},
\begin{equation}\label{eq:adddeg}
\deg\biggl(a_{l_0}\sum_{j=1}^{n^r}w_jb_j\biggl)=r
\end{equation}
since some \(b_j\) is nonzero. It, then,
follows from \eqref{eq:fg} and \eqref{eq:adddeg} that \(\deg(fg)=s+r\).

\end{proof}

The following are immediate consequences of the theorem.

\begin{corollary}\label{cor:sfedom}
Let \(R\) be a ring, let \(n\) be a positive integer,
let \(\sigma\colon R\to M_n(R)\) be a megainjective ring homomorphism, let 
\(\delta\colon R\to R^n\) be a right \(\sigma\)-derivation and let 
\(S=R\langle x_1,\dots,x_n;\sigma,\delta\rangle\) be the \((\sigma,\delta)\)-free skew extension 
of \(R\). Then
\renewcommand{\labelenumi}{(\roman{enumi})}
\begin{enumerate}
    \item \(S\) is a domain;
    \item \(U(S)=U(R)\);
    \item \(S\) is semiprimitive.\qed
\end{enumerate}
\end{corollary}

\begin{proof}
Items (i) and (ii) are clear consequences of Theorem~\ref{theo:degreeisadditive}. For (iii), let 
 \(f\in J(S)\), the Jacobson radical
of \(S\), and set \(b=1-x_1f\). By (ii), \(1-f, b\in U(R)\). So \(f\in R\) and
\(x_1f=1-b\in R\). If \(f\neq 0\), then \(x_1f\neq 0\), by (i), which 
implies \(1=\deg(x_1f)=\deg(1-b)=0\). It follows that \(f=0\), therefore \(S\) is semiprimitive.
\end{proof}

When no \(\sigma\)-derivation is present, megainjectivity of
\(\sigma\) is also a necessary conditions for a skew free extension to be a domain,
as the next result shows.

\begin{corollary}
    Let \(R\) be a ring, let \(n\) be a positive integer and let \(\sigma \colon R\to M_n(R)\)
		be a ring homomorphism.
 The following statements are equivalent:
\renewcommand{\labelenumi}{(\alph{enumi})}
\begin{enumerate}
  \item \(\sigma\) is megainjective.
	\item \(\deg\) is a degree-function on \(R\langle x_1,\dots, x_n; \sigma,\delta\rangle\) 
	  for any right \(\sigma\)-derivation \(\delta\).
  \item \(R\langle x_1,\dots, x_n; \sigma,0\rangle\) is a domain, where \(0\) 
	  denotes the zero \(\sigma\)-derivation. 
   \end{enumerate}

In particular, if \(R\langle x_1,\dots, x_n; \sigma,0\rangle\) is a domain, then so is 
\(R\langle x_1,\dots, x_n; \sigma,\delta\rangle\) for any right \(\sigma\)-derivation \(\delta\).
\end{corollary}
\begin{proof}
That (a) and (b) are equivalent is the content of Theorem~\ref{theo:degreeisadditive};
 and that (a) implies (c) follows from Corollary~\ref{cor:sfedom}.

Let us show that (c) implies (a). 
Suppose that \(S=R\langle x_1,\dots, x_n
				; \sigma,0\rangle\) is a domain. If
 \(\sigma\) were not megainjective, there would exist \(r\geq 1\) and \(a\in R\setminus\{0\}\) 
such that
the columns of \(\sigma^{(r)}(a)\) are right \(R\)-linearly dependent.
Since \(\delta=0\), \eqref{eq:equivalence} becomes
\[
a\biggl(\sum_{j=1}^{n^r}w_jb_j\biggr)=0,
\quad\text{if, and only if,}\quad
L_1b_1+L_2b_2+\dots+L_{n^r}b_{n^r}=0.
\]
So non-megainjectivity of  \(\sigma\) implies that \(S\) is not a domain. Hence,
(a) follows from (c).
\end{proof}

For a given \(\sigma\)-derivation \(\delta\), 
the filtration \(\{S_m\}_{m\geq 0}\) on the skew free extension
\(S=R\langle x_1,\dots,x_n;\sigma,\delta\rangle\), induced by the degree function
\(\deg\), is such that its associated graded ring is clearly isomorphic to the zero derivation
skew free
extension \(R\langle x_1,\dots,x_n;\sigma,0\rangle\).
Therefore, the fact that (c) implies (a) above can also be obtained as 
a consequence of Theorem~\ref{theo:degreeisadditive} in view of
\cite[Proposition~2.6.1]{pC1995}. This relation between 
\(R\langle x_1,\dots,x_n;\sigma,\delta\rangle\) and \(R\langle x_1,\dots,x_n;\sigma,0\rangle\)
will be explored again in Section~\ref{sec:prime}.

Recall, from Example~\ref{ex:dr}, that if \(R\) is a division ring, then
any ring homomorphism \(\sigma\colon R\to M_n(R)\) will be megainjective. Therefore, by 
Corollary~\ref{cor:sfedom}, any skew free
extension of a division ring is a domain. As stated in \cite[Th.~2.8]{LM2024}, more is
true: any skew free extension of a division ring is, in fact, a free ideal
ring (because it satisfies Cohn's weak algorithm with respect to the degree-function
\(\deg\)) and, as so, has a universal division ring of fractions.

\section{Series}\label{sec:series}

In this section, we exhibit a sufficient condition on the right \(\sigma\)-derivation
\(\delta\) for a \((\sigma,\delta)\)-free skew extension to be embeddable into a
``power series'' ring.

\medskip

Let \(R\) be a ring, let \(n\) be a positive integer, let \(\sigma\colon R\to M_n(R)\) be a ring homomorphism, let \(\delta\colon R\to R^n\) be a right \(\sigma\)-derivation and consider
		the \((\sigma,\delta)\)-free skew extension \(S=R\langle x_1,\dots,x_n;\sigma,\delta
		\rangle\)
of \(R\) generated by \(X=\{x_1,\dots,x_n\}\).
Given \(f\in S\), there are unique \(a_w\in R\), almost all zero, such that
\[
f=\sum_{w\in \langle X\rangle} wa_w.
\]
If \(f\neq 0\), the \emph{order} of \(f\) is defined to be the non-negative integer
\(\ord(f)\) given by
\[
\ord(f) = \min\{\lvert w\rvert \mid a_w\neq 0\}.
\]

Given \(f\in S\), it is clear that, for all \(j=1,\dots,n\), either \(fx_j=0\) or
\(\ord(fx_j)\geq\ord(f)\). It, then, follows by induction on \(\lvert w\rvert\) that
if \(fw\neq 0\) then
\begin{equation}\label{eq:order}
\ord(fw)\geq\ord(f),
\end{equation} 
for all \(w\in \langle X\rangle\).

The main concept of this section is the following.

\begin{definition}
Let \(R\) be a ring, let \(n\) be a positive integer and let \(\sigma\colon R\to M_n(R)\) be a ring homomorphism. A right \(\sigma\)-derivation \(\delta\colon R\to R^n\) is said
to be \emph{locally nilpotent} if for every \(a\in R\), there exists a positive integer
\(p=p(a)\) such that \((\delta_{j_p}\dotsm\delta_{j_1})(a)=0\) for all
\(j_1,\dots,j_p\in\{1,\dots,n\}\).
\end{definition}

The following are examples of locally nilpotent derivations.

\begin{example}
Given a ring homomorphism \(\sigma\colon R\to M_n(R)\) and \(c\in R^n\), the map
\(\delta\colon R\to R^n\), given by \(\delta(a)=ac-c\sigma(a)\), for all \(a\in R\), 
is a right \(\sigma\)-derivation (cf.~\cite[Example~4]{MK2019}), known as the \emph{inner} 
right \(\sigma\)-derivation defined by \(c\).

Suppose that \(R\) has a nilpotent ideal \(I\) such that \(\sigma_{ij}(I)\subseteq I\),
for all \(i,j=1,\dots,n\); moreover suppose that
 \(c=[c_1 \ c_2 \ \dots \ c_n]\in R^n\) is such that 
\(c_i\in I\), for all \(i=1,\dots,n\). Then the inner 
right \(\sigma\)-derivation \(\delta\) defined by \(c\) is locally nilpotent, for, in this
case, \(\sigma_{ij}(I^r)\subseteq I^r\) and, therefore, 
\((\delta_{i_r}\dotsm\delta_{i_1})(a)\in I^r\) for all \(a\in R\) and \(r\geq 1\).
 \end{example}

\begin{example}
Let \(R=k[t_1,\dots, t_n]\) be the polynomial algebra on \(n\) indeterminates over 
a field \(k\) of zero characteristic and let \(\sigma\colon R \to M_n(R)\) be the
scalar ring homomorphism.
The right \(\sigma\)-derivation \(\delta\colon R\to R^n\) such
that \(\delta_j= \frac{\partial}{\partial t_j}\) , for all \(j=1,\dots,n\), is locally nilpotent,
since, for any monomial \(f\in R\) of degree \(m\), \((\delta_{j_{m+1}}\dotsm\delta_{j_1})(f)=0\).
\end{example}

Let \(R\) be a ring, let \(n\) be a positive integer \(n\), let \(\sigma\colon R\to M_n(R)\) 
be a ring homomorphism, let \(\delta\colon R\to R^n\) be a right \(\sigma\)-derivation 
and let \(S=R\langle x_1,\dots,x_n;\sigma,\delta\rangle\) be the \((\sigma,\delta)\)-free skew extension of \(R\) generated by \(X=\{x_1,\dots,x_n\}\). We shall show that 
the ideas contained in \cite[Section 5]{GSZ2019} can be adapted to prove that a free skew 
series ring containing \(S\)
can be constructed provided that \(\delta\) is locally nilpotent.

\begin{lemma}\label{le:ord1}
For all \(a\in R\), \(k\geq 1\) and \(j_1,\dots,j_k\in\{1,\dots,n\}\),
\[
ax_{j_1}\dotsm x_{j_k} \equiv (\delta_{j_k}\dotsm\delta_{j_1})(a) \ \dpmod{I},
\]
where \(I\) denotes the right ideal of \(S\) generated by \(X\).
\end{lemma}

\begin{proof} By induction on \(k\), the base case being \eqref{x_ia}.
\end{proof}

\begin{lemma}\label{le:ord2}
Suppose \(\delta\) is locally nilpotent and let \(a\in R\). Then, for every
\(q\geq 0\), there exists a positive integer \(N_q= N_q(a)\) such that for
all \(w\in \langle X\rangle\) with \(\lvert w\rvert\geq N_q\), either
\(aw=0\) or \(\ord(aw)>q\).
\end{lemma}

\begin{proof}
In view of \eqref{eq:order}, it is sufficient to prove that the condition 
holds for \(w\in \langle X\rangle\) with
\(\lvert w\rvert = N_q\).

The integers \(N_q\) are constructed inductively. For \(q=0\), take \(N_0\) such
that \((\delta_{j_{N_0}}\dotsm \delta_{j_1})(a)=0\), for all \(j_1,\dots, j_{N_0} \in
\{1,\dots, n\}\). Given \(w\in \langle X\rangle\)
with \(\lvert w\rvert=N_0\) and \(aw\neq 0\), it follows from Lemma~\ref{le:ord1} that
\(\ord(aw)>0\).

Suppose that \(N_q\) has been defined. For each \(w\in \langle X\rangle_{N_q}\), write
\[
aw = \sum_{v\in\langle X\rangle}vc_v^w,
\]
with \(c_v^w\in R\). Since there are only finitely many nonzero elements in the
set \(\{c_v^w \mid v\in\langle X\rangle, w\in\langle X\rangle_{N_q}\}\)
there exists \(p\geq 1\) such that
\((\delta_{j_p}\dotsm\delta_{j_1})(c_v^w)=0\) for all \(v\in\langle X\rangle\), 
\(w\in \langle X\rangle_{N_q}\) and \(j_1,\ldots,j_p \in \{1,\ldots,n\}\). Let \(N_{q+1}=p+N_q\). Take \(u\in \langle X\rangle\)
with \(\lvert u\rvert=N_{q+1}\) such that \(au\neq 0\). Write \(u=u_1u_2\) with
\(u_1\in \langle X\rangle_{N_q}\) and \(u_2\in\langle X\rangle_p\). Then,
\[
au = (au_1)u_2= \sum_{v\in\langle X\rangle} vc_v^{u_1}u_2.
\]
For each \(v\) such that \(c_v^{u_1}u_2\neq 0\), we have, on the one hand, that
\(\lvert v\rvert>q\), by induction hypothesis, and, on the other hand, that
\(\ord(c_v^{u_1}u_2)\geq 1\), by Lemma~\ref{le:ord1}. It follows that
\(\ord(au)>q+1\).
\end{proof}

Now let \(R\langle\langle x_1,\dots,x_n;\sigma,\delta\rangle\rangle\) denote the
set of all formal sums
\[
\sum_{w\in\langle X\rangle} wa_w,
\]
with \(a_w\in R\), possibly infinitely many nonzero. This set contains the
skew free extension \(S=R\langle x_1,\dots,x_n;\sigma,\delta\rangle\) and has a natural
additive abelian group structure extending the one in \(S\). The next result shows
that \(R\langle\langle x_1,\dots,x_n;\sigma,\delta\rangle\rangle\) can be naturally
made into a ring if \(\delta\) is locally nilpotent.

\begin{theorem}
Let \(R\) be a ring, let \(n\) be a positive integer, let \(\sigma\colon R\to M_n(R)\) 
be a ring homomorphism, let \(\delta\colon R\to R^n\) be a right \(\sigma\)-derivation 
and let \(S=R\langle x_1,\dots,x_n;\sigma,\delta\rangle\) be the \((\sigma,\delta)\)-free skew extension of \(R\) generated by \(x_1,\dots,x_n\). If \(\delta\) is locally
nilpotent, then the natural extension of the multiplication from \(S\) to 
\(R\langle\langle x_1,\dots,x_n;\sigma,\delta\rangle\rangle\) is well-defined.
\end{theorem}

\begin{proof}
Take \(f,g\in R\langle\langle x_1,\dots,x_n;\sigma,\delta\rangle\rangle\), say
\(f=\sum_{v\in\langle X\rangle}vb_v\) and \(g=\sum_{w\in \langle X\rangle}wc_w\),
with \(b_v,c_w\in R\). We want to show that \(fg = \sum_{u\in \langle X\rangle}ua_u\),
where, for each \(u\in \langle X\rangle\), \(a_u\) is an element of \(R\), obtained by
taking (finite) sums of products of elements of \(R\).

Fix \(u\in\langle X\rangle\) with \(\lvert u\rvert = q\) and let us consider the
contribution to \(a_u\) coming from each term of the form \(vb_vwc_w\). If
\(vb_vwc_w = 0\) for all \(v,w\in\langle X\rangle\), then \(a_u=0\). Otherwise,
let \(v,w\in\langle X\rangle\) be such that \(vb_vwc_w\neq 0\). On the one hand,
if \(\lvert v\rvert > q\), then \(\ord(vb_vwc_w)>q\) and, hence, this term
does not contribute to \(a_u\). On the other hand, if \(\lvert w\rvert \geq N_q(b_v)\),
then \(\ord(b_vw)>q\), by Lemma~\ref{le:ord2}; so \(\ord(vb_vwc_w)>q\) and, again,
\(vb_vwc_w\) does not contribute to \(a_u\). Hence, the only terms of the form
\(vb_vwc_w\) that can, eventually, contribute to \(a_u\) are the ones 
with \(\lvert v\rvert \leq q\)
and \(\lvert w\rvert<s(q) = \max\{N_q(b_v) \mid v\in \langle X\rangle, \lvert v\rvert\leq q\}\). 
It follows
that \(a_u\) coincides with the coefficient of \(u\) in the following element of \(S\):
\[
\sum_{\substack{v,w\in \langle X \rangle\\ 
\lvert v\rvert\leq q, \lvert w\rvert < s(q)}} vb_vwc_w.
\]
Therefore, \(a_u\in R\).
\end{proof}

\section{Primeness}\label{sec:prime}

Here we shall show that a condition analogous to the one described in
\cite[Theorem~4.4]{LLM1997} for a skew polynomial ring to be prime holds in the
case of ``triangular'' skew free extensions.

\medskip

Let \(R\) be a ring and let \(n\) be a positive integer. Throughout this section we will 
focus exclusively on \emph{upper triangular} ring homomorphisms \(\sigma\colon R\to M_n(R)\),
that is, homomorphisms for which \(\sigma_{ij}\colon R\to R\) are zero maps for all
\(i>j\). In this case, it easy to see that \(\sigma_{ii}\) are ring endomorphisms of
\(R\), for all \(i=1,\dots,n\). Also, if \(\delta\colon R\to R^n\) is a right 
\(\sigma\)-derivation, then in \(S=R\langle x_1,\dots,x_n;\sigma,\delta\rangle\), one has
\[
ax_j = \sum_{i=1}^jx_i\sigma_{ij}(a) + \delta_j(a),
\]
for all \(j=1,\dots,n\) and \(a\in R\).

We start by defining a total ordering on the set \(\langle X\rangle\) of all monomials
on \(X=\{x_1,\dots,x_n\}\) that extends the ordering on \(\langle X\rangle_r\) used
in Section~\ref{sec:domain}. Once fixed that \(x_1<x_2<\dots<x_n\), let
\(u,v\in \langle X\rangle\setminus\{1\}\) and define \(u<v\) if \(\lvert u\rvert < 
\lvert v\rvert\) or \(\lvert u\rvert =\lvert v\rvert\) and \(u\) precedes \(v\) lexicographically,
that is, \(u=x_{i_1}\dotsm x_{i_r}, v=x_{j_1}\dotsm x_{j_r}\) and \(x_{i_1}<x_{j_1}\) or 
there exists \(s\in\{2,\dots,r\}\) such that
\(x_{i_1}=x_{j_1}, x_{i_2}=x_{j_2},\dots, x_{i_{s-1}}=x_{j_{s-1}}\) and \(x_{i_s}<x_{j_s}\).
Note that since \(\lvert 1\rvert = 0\), we have \(1<w\), for all \(w\in\langle X\rangle\),
\(w\neq 1\). Because for every positive integer \(n\) there are only finitely many
monomials of length \(n\), the set \(\langle X\rangle\) is, in fact, well-ordered by \(<\):
the minimal element of a nonempty subset \(U\) of \(\langle X\rangle\) is the smallest
monomial with minimal length in \(U\).
Moreover, \(<\) is a monomial ordering in the sense that if \(u<v\), then
\(w_1uw_2 < w_1vw_2\) for all \(w_1,w_2\in\langle X\rangle\).

For every \(w\in\langle X\rangle\), let \(S_{<w}\) be the right \(R\)-submodule of
\(S\) generated by \(\{v\in\langle X\rangle \mid v < w\}\). It is clear that
\(uS_{<w}\subseteq S_{<uw}\) for all \(u,w\in\langle X\rangle\). When \(\sigma\) is
upper triangular, we get a symmetric inclusion, as we shall see in Lemma~\ref{le:prime2}.

Given \(0\neq f\in S\), say \(f=\sum_{w\in\langle X\rangle} wa_w\), with \(a_w\in R\),
the monomial \(\max\{w\in\langle X\rangle \mid a_w\neq 0\}\) is called the \emph{leading monomial} 
of \(f\). And, of course, if \(w\) is the leading monomial of \(f\), then we can
write \(f=wa+g\), with \(a\in R\setminus\{0\}\) and \(g\in S_{<w}\).

The following is a refinement of Lemma~\ref{lem:scalartimesword} for the case of
an upper triangular \(\sigma\).

\begin{lemma}\label{le:prime1}
Suppose that \(\sigma\) is upper triangular and let \(a\in R\), \(w\in\langle X\rangle\).
Then,
\[
aw \equiv w\sigma_w(a) \ \dpmod{S_{<w}},
\]
where if \(w=x_{j_1}\dotsm x_{j_r}\) then \(\sigma_w = \sigma_{j_rj_r}\dotsm
\sigma_{j_2j_2}\sigma_{j_1j_1}\), and \(\sigma_1=\Id\).
\end{lemma}

\begin{proof}
The result is clearly valid for \(w=1\). Let \(w=x_{j_1}\dotsm x_{j_r}\), with \(r\geq 1\).
We proceed by induction on \(r\). For \(r=1\), since \(\sigma\) is upper triangular, we have
\[
ax_{j_1} = \sum_{i=1}^{j_1}x_i\sigma_{ij_1}(a) + \delta_{j_1}(a) = 
x_{j_1}\sigma_{j_1j_1}(a) + \sum_{i=1}^{j_1-1}x_i\sigma_{ij_1}(a) + \delta_{j_1}(a).
\]
So, \(ax_{j_1}-x_{j_1}\sigma_{x_{j_1}}(a)\in S_{<x_{j_1}}\).
Suppose the result valid for \(r\geq 1\) and let \(w=x_{j_1}\dotsm x_{j_r}\).
Then, there exists \(\sum_{v<w}
vb_v\in S_{<w}\), with \(b_v\in R\), such that
\begin{multline*}
awx_{j_{r+1}} = wx_{j_{r+1}} \sigma_{wx_{j_{r+1}}}(a)  + \sum_{i=1}^{j_{r+1}-1}
wx_i\sigma_{ij_{r+1}}(\sigma_w(a)) \\ + w\delta_{j_{r+1}}(\sigma_w(a)) + 
\sum_{v<w}\sum_{i=1}^{j_{r+1}} vx_i\sigma_{ij_{r+1}}(b_v) + \sum_{v<w}v\delta_{j_{r+1}}(b_v).
\end{multline*}
So, indeed, \(awx_{j_{r+1}} - wx_{j_{r+1}} \sigma_{wx_{j_{r+1}}}(a) \in S_{<wx_{j_{r+1}}}\).
\end{proof}

\begin{lemma}\label{le:prime2}
Suppose \(\sigma\) is upper triangular, and let \(v,w\in\langle X\rangle\). Then 
\renewcommand{\labelenumi}{(\roman{enumi})}
\begin{enumerate}
\item \(RS_{<w}\subseteq S_{<w}\),
\item \(S_{<v}w\subseteq S_{<vw}\),
\item \(S_{<v}S_{<w}\subseteq S_{<vw}\).
\end{enumerate}
\end{lemma}

\begin{proof}
For (i), let \(a\in R\) and \(f=\sum_{u<w}ub_u\in S_{<w}\), with \(b_u\in R\). Then,
for each \(u\in\langle X\rangle\), by Lemma~\ref{le:prime1}, there exists \(f_u\in S_{<u}\)
such that \(au=u\sigma_u(a)+f_u\). Hence,
\[
af = \sum_{u<w}u\sigma_u(a)b_u + \sum_{u<w}f_ub_u\in S_{<w}.
\]

For (ii), let \(g=\sum_{u<v}ud_u\in S_{<v}\), with \(d_u\in R\). Then, 
for each \(u\in\langle X\rangle\) with \(u<v\), by Lemma~\ref{le:prime1},
there exists \(h_u\in S_{<w}\) such that \(d_uw = w\sigma_w(d_u)+h_u\).
Then,
\[
gw = \sum_{u<v} uw\sigma_w(d_u) + \sum_{u<v} uh_u.
\]
Since \(u<v\) implies \(uw<vw\), we have \(uh_u\in uS_{<w}\subseteq S_{<uw}\subseteq S_{<vw}\).
So, \(gw\in S_{<vw}\).

Inclusion (iii) follows from (ii).
\end{proof}

We now show that a result analogous to \cite[Lemma~4.2]{LLM1997} holds in this
more general setting. A word on notation, given a ring \(R\) and a subset
\(U\) of \(R\), we denote the \emph{right (respectively, left) annihilator} of \(U\) in \(R\) by
\(\rann_R(U) = \{z\in R \mid Uz=0\}\) (respectively, \(\lann_R(U)=\{z\in R\mid zU=0\}\)).

\begin{lemma}\label{le:prime3}
Suppose that \(\sigma\) is upper triangular and that \(\sigma_{ii}\in\Aut(R)\) for
all \(i=1,\dots,n\). Let \(p=wa + f\in S\), with \(w\in\langle 
X\rangle\), \(a\in R\)  and \(f\in S_{<w}\).
\renewcommand{\labelenumi}{(\roman{enumi})}
\begin{enumerate}
    \item If \(\rann_R(a)\subseteq \rann_R(p)\), then \(\rann_S(p)\subseteq \rann_S(a)\).
		\item If \(\lann_R\bigl(\sigma_w^{-1}(a)\bigr) \subseteq \lann_R(p)\), then
		  \(\lann_S(p)\subseteq \lann_S(\sigma_w^{-1}(a))\).
\end{enumerate}
\end{lemma}

\begin{proof}
We prove only (ii), the proof of (i) being similar. Suppose that
\(\lann_S(p)\not\subseteq\lann_S\bigl(\sigma_w^{-1}(a)\bigr)\) and take 
\(h\in\lann_S(p)\setminus\lann_S\bigl(\sigma_w^{-1}(a)\bigr)\) of the form
\(h=uc+g\), with \(u\in\langle X\rangle\), \(c\in R\setminus\{0\}\) and \(g\in S_{<u}\) 
such that \(u\) is minimal among the leading monomials of the elements of 
\(\lann_S(p)\setminus\lann_S\bigl(\sigma_w^{-1}(a)\bigr)\). Then
\[
0=hp = (uc+g)(wa+f) = ucwa + ucf + gwa + gf.
\]
By Lemma~\ref{le:prime2}, \(ucf + gwa + gf\in S_{<uw}\) and, by Lemma~\ref{le:prime1},
\(ucwa \equiv uw\sigma_w(c)a \pmod{S_{<uw}}\). It follows that \(\sigma_w(c)a=0\) and, hence,
\(c\sigma_w^{-1}(a)=0\). By the hypothesis, \(cp=0\). But this implies that \(h-uc\in 
\lann_S(p)\setminus\lann_S\bigl(\sigma_w^{-1}(a)\bigr)\), contradicting the minimality of \(u\).
\end{proof}

\begin{corollary}\label{cor:prime}
Suppose that \(\sigma\) is upper triangular and that \(\sigma_{ii}\in\Aut(R)\) for
all \(i=1,\dots,n\). Let \(I\) be a nonzero right (respectively, left) \(R\)-submodule of \(S\) and
let \(p=wa+f\in I\), with \(w\in\langle X\rangle\), \(a\in R\setminus\{0\}\) and \(f
\in S_{<w}\) such that \(w\) is minimal among the leading monomials of nonzero elements of \(I\). 
Then \(a\rann_S(p)=0\) 
(respectively, \(\lann_S(p)\sigma_w^{-1}(a)=0\)).
\end{corollary}

\begin{proof}
Take \(b\in\rann_R(a)\). Then, \(pb\in I\) and \(pb=(wa+f)b=fb\). By the minimality of
\(w\), one must have \(pb=0\). Hence, by Lemma~\ref{le:prime3}(i), \(\rann_S(p)\subseteq
\rann_S(a)\). So, \(a\rann_S(p)=0\).
\end{proof}

\begin{theorem}\label{th:prime}
Let \(R\) be a ring, let \(n\) be a positive integer, let \(\sigma\colon R\to M_n(R)\) 
be a ring homomorphism, let \(\delta\colon R\to R^n\) be a right \(\sigma\)-derivation 
and let \(S=R\langle x_1,\dots,x_n;\sigma,\delta\rangle\) be the \((\sigma,\delta)\)-free skew extension of \(R\) generated by \(x_1,\dots,x_n\). Suppose that \(\sigma\) is
upper triangular and that \(\sigma_{ii}\in\Aut(R)\) for all \(i=1,\dots,n\). Then, the
following statements are equivalent.
\renewcommand{\labelenumi}{(\alph{enumi})}
\begin{enumerate}
    \item \(S\) is a prime ring.
		\item \(aSb\neq 0\), for all \(a,b\in R\setminus\{0\}\).
\end{enumerate}
\end{theorem}

\begin{proof}
Let us prove that (b) implies (a), the other implication being obvious. Suppose \(S\) is
not prime and let \(I\) be a nonzero ideal of \(S\) such that \(J=\rann_S(I)\neq 0\). Let
\(p=wa+f\in J\), with \(w\in\langle X\rangle\), \(a\in R\setminus\{0\}\) and \(f
\in S_{<w}\) such that \(w\) is minimal among the leading monomials of nonzero elements of 
\(J\). Because \(I\subseteq \lann_S(J)\subseteq \lann_S(p)\),
by Corollary~\ref{cor:prime}, \(I\sigma_w^{-1}(a)=0\), that is, \(\sigma_w^{-1}(a)\in
\rann_S(I)\cap R=J\cap R\). So \(J\cap R\neq 0\). Since \(I\) is a right \(R\)-submodule
of \(S\), another application of Corollary~\ref{cor:prime} gives \(\lann_S(J)\cap R\neq 0\).
Taking \(0\neq b\in \lann_S(J)\cap R\) and \(0\neq c\in J\cap R\) one gets \(bSc=b(Sc)\subseteq bJ=0\), contradicting (b).
\end{proof}

We remark that Theorem~\ref{th:prime} holds more generally for triangularizable 
homomorphisms \(\sigma\colon R\to M_n(R)\). More precisely, suppose that there
exists an invertible \(P\in M_n(R)\) such that, the map
\[
\begin{array}{rcl}
\tau\colon R & \longrightarrow & M_n(R)\\
a & \longmapsto & P\sigma(a)P^{-1}
\end{array}
\]
is a ring homomorphism such that \(\tau_{ij}\) is the zero map for all \(i>j\) and
\(\tau_{ii}\in\Aut(R)\), for all \(i=1,\dots,n\). Given a right \(\sigma\)-derivation
\(\delta \colon R\to R^n\), the map \(\rho\colon R\to R^n\) defined by \(\rho(a) =
\delta(a)P^{-1}\), for all \(a\in R\), is a right \(\tau\)-derivation. Let
\(S=R\langle x_1,\dots,x_n;\sigma,\delta\rangle\) 
and \(S'=R\langle y_1,\dots, y_n;\tau,\rho\rangle\).
By
\cite[Corollary~21]{uM2020}, there exists an \(R\)-ring isomorphism 
\(\varphi\colon S\to S'\).
Thus, if \(aSb\neq 0\) for all \(a,b\in R\setminus\{0\}\), then, applying \(\varphi\),
one obtains that \(a S'b\neq 0\) for all \(a,b\in R\setminus\{0\}\). By 
Theorem~\ref{th:prime}, \(S'\) is prime and, so, \(S\) is prime.

For a not necessarily triangular homomorphism \(\sigma\) and an arbitrary
\(\sigma\)-derivation \(\delta\), recall that, as we have seen in Section~\ref{sec:domain}, 
\(R\langle x_1,\dots,x_n;\sigma,0\rangle\) can be regarded as the associated graded
ring of \(R\langle x_1,\dots,x_n;\sigma,\delta\rangle\) with respect to the
filtration defined by the pseudo-valuation \(\deg\). Hence, using \cite[Lemma~II.3.2.7]{HvO1996},
the following
reduction can be made:

\begin{proposition}
Let \(R\) be a ring, let \(n\) be a positive integer and let \(\sigma\colon R\to M_n(R)\) 
be a ring homomorphism. If \(R\langle x_1,\dots,x_n;\sigma,0\rangle\) is a prime 
(respectively, semiprime) ring,
then so is \(R\langle x_1,\dots,x_n;\sigma,\delta\rangle\) for
any right \(\sigma\)-derivation \(\delta\colon R\to R^n\). \qed
\end{proposition}

%

\end{document}